\documentclass[11pt,a4paper]{amsart}
\usepackage{times}
\usepackage{amscd,amssymb,amsthm,latexsym,amsmath}
\usepackage{amsfonts,latexsym,amssymb,url}
\usepackage{hyperref}
\usepackage{pifont}

\newtheorem{theorem}{Theorem}[section]

\newtheorem{lemma}{Lemma}[section]

\newtheorem{defi}{Definition}[section]
\newtheorem{remark}{Remark}[section]
\newcommand{\med}{{\int\!\!\!\!\!\!\!\!-\!\! -\!\!\!\!}}
\newcommand{\R}{{\mathbb R}}

\newcommand{\N}{{\mathbb N}}

\begin{document}
   \title[H\"older regularity for non divergence form elliptic equations]
   { H\"older regularity for non divergence form elliptic equations with discontinuous coefficients }
   \author{G. Di Fazio}
   \address{Dipartimento di Matematica e Informatica\\ Universit\`a di Catania\\
   Viale A. Doria 6, 95125, Catania, Italy}
   \email{difazio@dmi.unict.it}
\author{M. S. Fanciullo}
   \address{Dipartimento di Matematica e Informatica\\ Universit\`a 
   di Catania\\
   Viale A. Doria 6, 95125, Catania, Italy}
   \email{fanciullo@dmi.unict.it}

    \author{P. Zamboni}
   \address{Dipartimento di Matematica e Informatica\\ Universit\`a 
   di Catania\\
   Viale A. Doria 6, 95125, Catania, Italy}
   \email{zamboni@dmi.unict.it}
   
   \thanks{}
\date{\today}

\linespread{1.3}

\begin{abstract}
In this note we study the global regularity in the Morrey spaces $L^{p,\lambda} $for the second derivatives for the strong solutions of non variational elliptic equations. 
\end{abstract}
\maketitle

\section{Introduction}

The aim of this note is to study the global Morrey regularity for the second derivatives of the strong solutions of non variational elliptic equations. 
Namely, given a bounded domain $\Omega\subset\R^n$, we consider the linear equation
$$
a_{ij}u_{x_ix_j} + b_i u_{x_i} + cu =f\,,
$$
where the lower order coefficients $b$, $c$ and $f$ are assumed in the Morrey space $L^{p,\lambda}$ ($1<p<+\infty$, $n-p<\lambda<n$) and the coefficients of the leading part are assumed in the class $VMO \cap L^\infty$.

As a consequence of our $W^{2,p}$ estimate (see Section 3) we obtain the H\"older continuity of the gradient.

The technique we use is quite simple; it is based on a multiplicative inequality for functions in Morrey classes
combined with an iterative procedure. 

The same problem has been studied in several papers by many Authors. Among them, we cite  \cite{caf1} and \cite{caf2} where, in the case $c=b=0$, Caffarelli proved that if $f$ belongs to the Morrey space $L^{n,n\alpha}$, with $0<\alpha<1$, then every $W^{2,p}$ - viscosity solution $u$ is of class $C^{1,\alpha}$. 
Subsequently, in \cite{dr} and \cite{dpr}, Caffarelli result were improved in a special case obtaining gradient regularity for any $p$.
The result in \cite{dr} was obtained via a representation formula for the second derivatives of the solutions used in \cite{cfl}  and the study of some non convolution type integral operators. The result in \cite{dpr} can be recovered from our by letting $b$ and $c$ identically zero.

\section{Preliminaries}
Let $\Omega$ be a bounded open set in $\R^n$ ($n\ge 3$). If $f\in L^1(\Omega)$ and $E\subset \Omega$ we set $f_E=\frac{1}{|E|}\int_E fdx$. 

We recall some classical definitions.
\begin{defi}
Let $1\le p<+\infty$, $0<\lambda<n$, $\Omega$ a bounded domain in $\R^n$. A function $f\in L^p(\Omega)$ belongs to the Morrey space $L^{p,\lambda}(\Omega)$ if
$$\|f\|^p_{L^{p,\lambda}(\Omega)}=\sup  r^{-\lambda}\int_{B_r(x_0)\cap\Omega}|f|^pdx<+\infty\,,$$
the supremum being taken over $x_0\in \Omega$ and $0<r\le{\rm diam}\Omega$.
\end{defi}
It is well known that $L^{p,\lambda}(\Omega)$ is a Banach space endowed with the above norm. 

\begin{defi}
Let $1\le p< +\infty$, $0<\lambda <n+p$, $\Omega$ a bounded domain in $\R^n$. A function $f\in L^p(\Omega)$ belongs to the Campanato space ${\mathcal L}^{p,\lambda}(\Omega)$ if
\begin{equation}\label{campanatodefinizione}
[f]^p_{{\mathcal L}^{p,\lambda}(\Omega)}=\sup  r^{-\lambda}\int_{B_r(x_0)\cap\Omega}|f-f_{B_r(x_0)\cap\Omega}|^p dx<+\infty\,,
\end{equation}
the supremum being taken over $x_0\in \Omega$ and $0<r\le{\rm diam}\Omega$.
\end{defi}

If $\lambda=n$ the definition gives back the definition of the space  $BMO$.

\begin{defi}
Let $\Omega$ be a bounded domain in $\R^n$. A function $f\in BMO (\Omega)$  belongs to $VMO(\Omega)$ if 
$$\eta(r)= \sup\med_{B_\rho(x_0)\cap\Omega}|f-f_{B_\rho(x_0)\cap \Omega}|dx$$
vanishes as $r\to 0^+$. Here the supremum is taken over $x_0 \in \Omega$ and $0<\rho<r$.
\end{defi}

The space ${\mathcal L}^{p,\lambda}$ is a Banach space endowed with the norm
$$ \|f\|_{{\mathcal L}^{p,\lambda}(\Omega)}=\|f\|_{L^p(\Omega)}+[f]_{{\mathcal L}^{p,\lambda}(\Omega)}\,,$$
where  $[f]_{{\mathcal L}^{p,\lambda}(\Omega)}$ is given by \eqref{campanatodefinizione}.

\begin{defi}
Let $\Omega$ be a bounded domain in $\R^n$.
We say that $\Omega$ satisfies the condition $K$ if there exists a positive constant $K>0$ such that 
$$|B_r(x_0)\cap\Omega|\ge K r^n\,,$$
 $\forall x_0\in \Omega$ and $0<r\le {\rm diam}\Omega$.
\end{defi}

\begin{remark}
Any Lipschitz domain satisfies the condition $K$.
\end{remark}

We state some useful results we need in the sequel.  

\begin{theorem}\label{cam}(\cite{cam}, Theorem 2.I). Let $\Omega\subset\R^n$ be a bounded domain satisfying condition $K$. 
Then 
\begin{itemize}
\item[1.] 
If $0<\lambda<n$, ${\mathcal L}^{p,\lambda}(\Omega)=L^{p,\lambda}(\Omega)$ and there exist two positive constants $C_1$ and $C_2$ such that
$$
C_1 [f]_{{\mathcal L}^{p,\lambda}}\le \|f\|_{L^{p,\lambda}}\le C_2\|f\|_{{\mathcal L}^{p,\lambda}}\,.
$$
\item[2.]
If $n<\lambda\le n+p$, ${\mathcal L}^{p,\lambda}(\Omega)=C^{0,\gamma}(\overline\Omega)$, with $\gamma =\frac{\lambda-n}{p}$, and there exist two positive constants $C_3$ and $C_4$ such that
$$C_3[f]_{{\mathcal L}^{p,\lambda}}\le [f]_{\gamma}\le C_4[f]_{{\mathcal L}^{p,\lambda}} \,,$$
where $[f]_{\gamma}= \sup\limits_{x\neq y}\dfrac{|f(x)-f(y)|}{|x-y|^\gamma}$.
\end{itemize}
\end{theorem}

In the sequel we will use the following multiplicative inequality.

\begin{lemma}\label{immersione} 
Let $u\in W^{1,p}(\Omega)$ and $f\in L^{p,\lambda}(\Omega)$, with $2\le p<n$, $n-p<\lambda<n$. If $\nabla u\in L^{p,\eta}(\Omega)$ for some $\eta\in[0,n-p[$, then 
$$fu\in L^{p,\lambda+\eta-n+p}(\Omega)\,.$$
Moreover there exists a positive constant $C$, independing of $u$ and $f$, such that 
$$
\|fu\|_{L^{p,\lambda+\eta-n+p}(\Omega)}\le C \|f\|_{L^{p,\lambda}(\Omega)}(\|\nabla u\|_{L^{p,\eta}(\Omega)}+\|u\|_{L^p(\Omega)})\,.
$$  
\end{lemma}

The lemma has been proved in \cite{d} (Lemma 4.1) for $p=2$. The extension to the case $p\neq 2$ is straightforward.

\section{ H\"older regularity} 

Let $\Omega\subset\R^n$ be a bounded domain satisfying the condition $K$. 

Let us consider the following linear second order elliptic equation in non divergence form
\begin{equation}\label{eq}
  a_{ij}u_{x_ix_j}+b_i u_{x_i}+cu=f\,,
\end{equation}
where we assume
\begin{equation}\label{ipo1}
a_{ij}(x)=a_{ji}(x) \quad a.e.\,\,x\in \Omega, \; i,j=1,\dots,n,
\end{equation}
\begin{equation}\label{ipo2}
 \exists \,\mu>0:\mu|\xi|^2\le \sum_{ij=1}^n a_{ij}(x)\xi_i\xi_j\le\frac{1}{\mu}|\xi|^2, \quad a.e.\,\,x\in \R^n,\,\forall \xi\in\Omega,
\end{equation}
and
\begin{equation}\label{ipo3} 
a_{ij}\in L^\infty( \Omega)\cap VMO(\Omega)\,,\, c,\, b_i, f\in L^p(\Omega).
\end{equation}

\begin{defi}
A function $u$ in $W^{2,p}(\Omega)$ is a strong solution of the equation \eqref{eq} if $u$ satisfies \eqref{eq} a.e. $x\in\Omega$.
\end{defi}

Here we recall the Theorem 3.3 in \cite{dpr} regarding the equation \eqref{eq} with $b_i=0$ and $c=0$. 

\begin{theorem}\label{teodpr}
Let \eqref{ipo1}, \eqref{ipo2}, \eqref{ipo3} hold true.
Let $u\in W^{2,p}\cap W^{1,p}_0(\Omega)$ be a strong solution of \eqref{eq1},  with $b_i=0$ and $c=0$ and $f$ belong to $L^{p,\lambda}(\Omega)$, $p>1$, $0<\lambda<n$. Then $D^2u\in L^{p,\lambda}(\Omega)$ and there exists a positive constant $C$ such that
\begin{equation*}
\|D^2u\|_{L^{p,\lambda}(\Omega)}\le C \left(\|u\|_{L^{p,\lambda}(\Omega)}+\|f\|_{L^{p,\lambda}(\Omega)}\right).
\end{equation*}

\end{theorem}
Our first result concerns the case $b_i=0$, $i=1,2,\dots,n$, i.e. we consider the equation
\begin{equation}\label{eq1}
  a_{ij}u_{x_ix_j}+cu=f\,.
\end{equation}

Regarding the equation \eqref{eq1} we prove the following result.

\begin{theorem}\label{teo1}
Let \eqref{ipo1}, \eqref{ipo2}, \eqref{ipo3} hold true and
let $u\in W^{2,p}\cap W^{1,p}_0(\Omega)$ be a strong solution of \eqref{eq1}, $c$, and $f$ belong to $L^{p,\lambda}(\Omega)$, $1<p<n$, $n-p<\lambda<n$. Then, for all $0<\epsilon<min\{n-p, p+\lambda-n\}$ we have that $\nabla u\in C^{0,\gamma}(\overline\Omega)$, where $\gamma =1- \frac{n-\lambda+\epsilon}{p}$. Moreover there exists a positive constant $C$ such that 
\begin{multline*}
[\nabla u]_{\gamma}
\le  C \left(\|c\|_{L^{p,\lambda}(\Omega)}\|\nabla u\|_{L^{p,n-p-\epsilon}(\Omega)}+ \right.\\
\left.+(\|c\|_{L^{p,\lambda}(\Omega)}+1)\|u\|_{L^{p,\lambda-\epsilon}(\Omega)}+\|f\|_{L^{p,\lambda}(\Omega)}\right) \,. 
\end{multline*}
\end{theorem}

\begin{proof}
We start by noting that $\nabla u\in W^{1,p}(\Omega)$ implies $\nabla u\in L^{p,p}(\Omega)$.

If $p\ge n-p$ then $\nabla u\in L^{p, n-p-\epsilon}(\Omega)$, with $0<\epsilon<min\{n-p, p+\lambda-n\}$, and we apply Lemma \ref{immersione} to obtain $cu\in L^{p, \lambda-\epsilon}(\Omega)$, from which and Theorem \ref{teodpr} we obtain $D^2u\in L^{p,\lambda-\epsilon}(\Omega)$.
Then, from Poincar\'e inequality $\nabla u\in {\mathcal L}^{p,p+\lambda-\epsilon}(\Omega)$ and since $p+\lambda-\epsilon>n$ we obtain that $\nabla u\in C^{0,\gamma}(\overline\Omega)$.

If $p<n-p$ from  Lemma \ref{immersione} we obtain  $cu\in L^{p, \lambda+2p-n}(\Omega)$, from which and Theorem \ref{teodpr}, since  $\lambda+2p-n<\lambda$  we obtain that $D^2u\in L^{p,\lambda+2p-n}(\Omega)$ 
and consequently $\nabla u\in {\mathcal L}^{p,\lambda+3p-n}(\Omega)$. 

If $\lambda+3p-n>n$ then $\nabla u \in C^{0,\gamma}(\overline\Omega)$. 

If $\lambda+3p-n<n$ $\nabla u\in L^{p,\lambda+3p-n}(\Omega)$. If also $\lambda+3p-n\ge n-p $ then $\nabla u\in L^{p, n-p-\epsilon}(\Omega)$, with $0<\epsilon<min\{n-p, p+\lambda-n\}$,   and we apply Lemma \ref{immersione} to obtain $cu\in L^{p, \lambda-\epsilon}(\Omega)$, from which and Theorem \ref{teodpr} we obtain $D^2u\in L^{p,\lambda-\epsilon}(\Omega)$.
Then $\nabla u\in {\mathcal L}^{p,p+\lambda-\epsilon}(\Omega)$ and since $p+\lambda-\epsilon>n$ we obtain that $\nabla u\in C^{0,\gamma}(\overline\Omega)$. 

If $\lambda+3p-n< n-p $  then from  Lemma \ref{immersione}  we obtain  $cu\in L^{p, 2\lambda+4p-2n}(\Omega)$, from which and Theorem \ref{teodpr},  since $2\lambda+4p-n<\lambda$,  we obtain that $D^2u\in L^{p,2\lambda+4p-2n}(\Omega)$ 
and  $\nabla u\in {\mathcal L}^{p,2\lambda+5p-2n}(\Omega)$. 

If $2\lambda+5p-2n>n$ then $\nabla u \in C^{0,\gamma}(\overline\Omega)$. 

If $2\lambda+5p-2n<n$ $\nabla u\in L^{p,2\lambda+5p-2n}(\Omega)$ we can proceed as in the case $\lambda+3p-n<n$ and so on.

Finally there exists $k\in\N$ such that $n-p\le(2k+1)p+k(\lambda-n)$,
and $\nabla u \in L^{p,(2k+1)p+k(\lambda-n)}(\Omega)$
then $\nabla u\in C^{0,\gamma}(\overline\Omega)$. Moreover from Theorem \ref{cam}, Theorem \ref{teodpr} and Lemma \ref{immersione}  we get
\begin{multline*}
[\nabla u]_{\gamma}\le C[\nabla u]_{{\mathcal L}^{p,p+\lambda-\epsilon}(\Omega)}\le C \|D^2u\|_{L^{p,\lambda-\epsilon}(\Omega)}\le\\
\le C \{\|cu\|_{L^{p,\lambda-\epsilon}(\Omega)}+\|u\|_{L^{p,\lambda-\epsilon}(\Omega)}+\|f\|_{L^{p,\lambda}(\Omega)}\}\le \\
\le C\{\|c\|_{L^{p,\lambda}(\Omega)}[\|\nabla u\|_{L^{p,n-p-\epsilon}(\Omega)}+\\
+\|u\|_{L^p(\Omega)}]+\|u\|_{L^{p,\lambda-\epsilon}(\Omega)}+\|f\|_{L^{p,\lambda}(\Omega)}\} \le\\
\le C \{\|c\|_{L^{p,\lambda}(\Omega)}\|\nabla u\|_{L^{p,n-p-\epsilon}(\Omega)}+\\
+(\|c\|_{L^{p,\lambda}(\Omega)}+1)\|u\|_{L^{p,\lambda-\epsilon}(\Omega)}+\|f\|_{L^{p,\lambda}(\Omega)}\} \,.
\end{multline*}


\end{proof}

Now we study 

\begin{equation}\label{eq2}
  a_{ij}u_{x_ix_j}+b_i u_{x_i}=f\,.
\end{equation}

\begin{theorem}\label{teo2}
Let \eqref{ipo1}, \eqref{ipo2}, \eqref{ipo3} hold true. 
Let $u\in W^{2,p}\cap W^{1,p}_0(\Omega)$ be a strong solution of \eqref{eq2}, $b_i$, $i=1,\dots,n$ and $f$ belong to $L^{p,\lambda}(\Omega)$, $p<n$, $n-p<\lambda<n$. Then there exists $k\in \N$ such that  $k\lambda-kn+(k+1)p>n$ and $\nabla u\in C^{0,\gamma}(\overline\Omega)$ with  $\gamma= \frac{k\lambda-(k+1)n+(k+1)p}{p}$. 
Moreover there exists a positive constant $C$ depending on $\|b\|_{L^{p,\lambda}(\Omega)}$ and $k$ such that
\begin{equation}\label{bfinale}
[\nabla u]_{\gamma}\le C\left(  \|D^2u\|_{L^{p}(\Omega)}
+\|u\|_{L^{p,\lambda-n+p}(\Omega)}
+\|f\|_{L^{p,\lambda}(\Omega)}\right)\,.
\end{equation}   

\end{theorem}
\begin{proof}
Since $D^2u\in L^p(\Omega)$ from Lemma \ref{immersione} we obtain that $b\cdot\nabla u\in L^{p, \lambda-n+p}(\Omega)$. Then from Theorem \ref{teodpr} $D^2u\in L^{p,\lambda-n+p}(\Omega)$, and we have
\begin{multline}\label{primopasso}
\|D^2u\|_{L^{p,\lambda-n+p}(\Omega)}\le \\
\le C\{\|b \nabla u\|_{L^{p,\lambda-n+p}(\Omega)}+\|u\|_{L^{p,\lambda-n+p}(\Omega)}+\|f\|_{L^{p,\lambda}(\Omega)} \}\le\\
\le C \{ \|b\|_{L^{p,\lambda}(\Omega)}[\|D^2u\|_{L^{p}(\Omega)}+\|u\|_{L^p(\Omega)}]+ \|u\|_{L^{p,\lambda-n+p}(\Omega)}+\|f\|_{L^{p,\lambda}(\Omega)}   \}\le\\
\le C \{  \|b\|_{L^{p,\lambda}(\Omega)}\|D^2u\|_{L^{p}(\Omega)}+\\
+( \|b\|_{L^{p,\lambda}(\Omega)}+1) \|u\|_{L^{p,\lambda-n+p}(\Omega)}+\|f\|_{L^{p,\lambda}(\Omega)}  \} \,.
\end{multline}

From Poincar\'e inequality $\nabla u\in {\mathcal L}^{p,\lambda-n+2p}(\Omega)$. 
If $\lambda-n+2p>n$ we have that $\nabla u\in C^{0,\gamma}(\overline\Omega)$ with $\gamma=\frac{\lambda-2n+2p}{p}$. So, from Theorem \ref{cam}, Poincar\'e inequality and \eqref{primopasso}
\begin{multline*}
[\nabla u]_\gamma\le [\nabla u]_{{\mathcal L}^{p,\lambda-n+2p}(\Omega)}\le \|D^2u\|_{L^{p,\lambda-n+p}(\Omega)}\le\\
\le C \{  \|b\|_{L^{p,\lambda}(\Omega)}\|D^2u\|_{L^{p}(\Omega)}+\\
+( \|b\|_{L^{p,\lambda}(\Omega)}+1) \|u\|_{L^{p,\lambda-n+p}(\Omega)}+\|f\|_{L^{p,\lambda}(\Omega)}  \} \,.
\end{multline*} 

If $\lambda-n+2p<n$, since $\lambda-n+p<n-p$, we can apply  Lemma \ref{immersione} to obtain $b\nabla u\in L^{p, 2\lambda-2n+2p}(\Omega )$. Since $2\lambda-2n+2p<\lambda$ from Theorem \ref{teodpr} $D^2u\in L^{p,2\lambda-2n+2p}(\Omega )$ and we get from also \eqref{primopasso} 
\begin{multline*}
\|D^2u\|_{L^{p,2\lambda-2n+2p}(\Omega )}\le \\
\le C\{\|b \nabla u\|_{L^{p,2\lambda-2n+2p}(\Omega)}+\|u\|_{L^{p,2\lambda-2n+2p}(\Omega)}+\|f\|_{L^{p,\lambda}(\Omega)} \}\le\\
\le C\{ \|b\|_{L^{p,\lambda}(\Omega)}[\|D^2u\|_{L^{p,\lambda-n+p}(\Omega)}+\|u\|_{L^p(\Omega)}]+\\
+ \|u\|_{L^{p,2\lambda-2n+2p}(\Omega)}+\|f\|_{L^{p,\lambda}(\Omega)}   \}\le\\
\le C\{  \|b\|^2_{L^{p,\lambda}(\Omega)}\|D^2u\|_{L^{p}(\Omega)}+( \|b\|^2_{L^{p,\lambda}(\Omega)}+\\
+\|b\|_{L^{p,\lambda}(\Omega)}+1) \|u\|_{L^{p,\lambda-n+p}(\Omega)}+(\|b\|_{L^{p,\lambda}(\Omega)}+1)\|f\|_{L^{p,\lambda}(\Omega)}  \}\,.
\end{multline*}   

Now from Poincar\'e inequality we obtain $\nabla u\in {\mathcal L}^{p, 2\lambda-2n+3p}(B_2)$. 

If $2\lambda-2n+3p>n$ $\nabla u\in C^{0,\gamma}(\overline\Omega)$, with $\gamma= \frac{2\lambda-3n+3p}{p}$, and 
\begin{multline*}
[\nabla u]_\gamma\le [\nabla u]_{{\mathcal L}^{p,2\lambda-2n+3p}(\Omega)}\le \|D^2u\|_{L^{p,2\lambda-2n+2p}(\Omega)}\le\\
\le C\{  \|b\|^2_{L^{p,\lambda}(\Omega)}\|D^2u\|_{L^{p}(\Omega)}+\\
+( \|b\|^2_{L^{p,\lambda}(\Omega)}+
+\|b\|_{L^{p,\lambda}(\Omega)}+1) \|u\|_{L^{p,\lambda-n+p}(\Omega)}+\\
+(\|b\|_{L^{p,\lambda}(\Omega)}+1)\|f\|_{L^{p,\lambda}(\Omega)}  \}\,.
\end{multline*}   

If $2\lambda-2n+3p<n$ ($2\lambda-2n+2p<n-p$) we proceed as the previous cases.
Finally there exists a positive integer $k$ such that $k\lambda-kn+(k+1)p>n$ then
$\nabla u\in C^{0,\gamma}(\overline\Omega)$ with $\gamma= \frac{k\lambda-(k+1)n+(k+1)p}{p}\,$  and 
\begin{multline*}
[\nabla u]_\gamma\le [\nabla u]_{{\mathcal L}^{p,k\lambda-kn+(k+1)p}(\Omega)}\le \|D^2u\|_{L^{p,k\lambda-kn+(k+1)p}(\Omega)}\le\\
\le C\{  \|b\|^k_{L^{p,\lambda}(\Omega)}\|D^2u\|_{L^{p}(\Omega)}+\\
+( \|b\|^k_{L^{p,\lambda}(\Omega)}+\|b\|^{k-1}_{L^{p,\lambda}(\Omega)}
+\dots+\|b\|_{L^{p,\lambda}(\Omega)}+1) \|u\|_{L^{p,\lambda-n+p}(\Omega)}+\\
+(\|b\|^{k-1}_{L^{p,\lambda}(\Omega)}+\dots+\|b\|_{L^{p,\lambda}(\Omega)}+1)\|f\|_{L^{p,\lambda}(\Omega)}  \}\,
\end{multline*} 
 from which \eqref{bfinale} follows.
\end{proof}

The techniques used in Theorems \ref{teo1} and \ref{teo2} allow also to prove the H\"older regularity for the gradient of the solutions $u\in  W^{2,p}\cap W^{1,p}_0(\Omega)$ of the complete equation \eqref{eq}.

\end{document}